\documentclass[11pt]{article}
\usepackage{amsmath}
\usepackage{amsfonts}
\usepackage{amssymb}
\usepackage{amsthm}
\usepackage{breqn}
\usepackage{setspace}
\usepackage{fullpage}
\usepackage{enumitem}
\usepackage{bbold} 
\usepackage{comment}
\usepackage{hyperref}
\usepackage{breqn}
\usepackage{todonotes}
\newtheorem{lemma}{Lemma}
\newtheorem{theorem}{Theorem} 
\newtheorem{corollary}{Corollary} 
\newtheorem{definition}{Definition}
\newtheorem{claim}{Claim}
\newtheorem{remark}{Remark}
\newtheorem{proposition}{Proposition}

\newtheorem{observation}{Observation}

\newcommand{\cH}{\mathcal{H}}

\newcommand{\E}{\mathcal{E}}

\newcommand{\F}{\mathcal{F}}

\newcommand{\h}{\mathcal{H}}

\newcommand{\lp}{\left (}
\newcommand{\rp}{\right )}

\newcommand{\abs}[1]{\left\lvert{#1}\right\rvert}
\newcommand{\floor}[1]{\left\lfloor{#1}\right\rfloor}
\newcommand{\ceil}[1]{\left\lceil{#1}\right\rceil}

\DeclareMathOperator{\Tr}{Tr}

\begin{document}
\title{Ramsey numbers of Berge-hypergraphs and related structures}

\author{
	Nika Salia\thanks{Alfr\'ed R\'enyi Institute of Mathematics     
		({\tt salia.nika@renyi.hu}).} \thanks{Central European University, Budapest.}\and	
	Casey Tompkins\thanks{Karlsruhe Institue of Technology, Germany ({\tt ctompkins496@gmail.com}).} \thanks{Discrete Mathematics Group, Institute for Basic Science (IBS), Daejeon, Republic of Korea.}\and
	Zhiyu Wang\thanks{University of South Carolina, Columbia, SC, 29208
		({\tt zhiyuw@math.sc.edu}).} 
	\and
	Oscar Zamora\thanks{Central European University, Budapest
		({\tt oscarz93@yahoo.es}).} \thanks{Universidad de Costa Rica, San Jos\'e.}
                }
\maketitle


\begin{abstract}
For a graph $G=(V,E)$, a hypergraph $\cH$ is called a \textit{Berge}-$G$, denoted by $BG$, if there is
an injection $i\colon V(G)\to V(\mathcal{H})$ and 
a bijection $f\colon E(G) \to E(\mathcal{H})$ such that for all $e=uv \in E(G)$, we have $\{i(u), i(v)\} \subseteq f(e)$. Let the Ramsey number $R^r(BG,BG)$ be the smallest integer $n$ such that for any $2$-edge-coloring of a complete $r$-uniform hypergraph on $n$ vertices, there is a monochromatic Berge-$G$ subhypergraph.
In this paper, we show that the 2-color Ramsey number of Berge cliques is linear. In particular, we show that $R^3(BK_s, BK_t) = s+t-3$ for $s,t \geq 4$ and $\max\{s,t\} \geq 5$ where $BK_n$ is a Berge-$K_n$ hypergraph. For higher uniformity, we show that $R^4(BK_t, BK_t) = t+1$ for $t\geq 6$ and $R^k(BK_t, BK_t)=t$ for $k \geq 5$ and $t$ sufficiently large. We also investigate the Ramsey number of trace
hypergraphs, suspension hypergraphs and expansion hypergraphs.
\end{abstract}


\section{Introduction}
Given a hypergraph $\h$, let $v(\h)$ denote the number of vertices of $\h$ and $e(\h)$ denote the number of hyperedges. We denote the sets of vertices and hyperedges of $\h$ by $V(\h)$ and $E(\h)$, respectively.  We say that a hypergraph is $r$-uniform if every hyperedge has size $r$. By $K_t^{(r)}$ we denote the $t$-vertex $r$-uniform clique (if $r=2$ we omit the superscript).  The set of the first $n$ integers is denoted by $[n]$, and for a set $S$, we denote by $\binom{S}{r}$ the set of $r$-element subsets of $S$. Furthermore we denote the power set of a set $S$ by $2^S$. For sets $A$ and $B$ we denote their disjoint union by $A\sqcup B$.

Ramsey theory is among the oldest and most intensely investigated topics in combinatorics.  It began with the seminal result of Ramsey from 1930.
\begin{theorem}[Ramsey~\cite{Ramsey}]
Let $r,t$ and $k$ be positive integers.  Then there exists an integer $N$ such that any coloring of the $N$-vertex $r$-uniform complete hypergraph with $k$ colors contains a monochromatic copy of the $t$-vertex $r$-uniform complete hypergraph.
\end{theorem}
Estimating the smallest value of such an integer $N$ (the so-called Ramsey number) is a notoriously difficult problem and only weak bounds are known.  Given the difficulty of this problem, many people began investigating variations of this problem where graphs other than the complete graphs are considered. An example of an early result in this direction due to Chv\'atal~\cite{TreeComplete} asserts that the Ramsey number of a $t$-clique versus any $m$-vertex tree is precisely $N=1 + (m-1)(t-1)$.  That is, any red-blue coloring of the complete graph $K_N$ yields a red $K_t$ or a blue copy of a given $m$-vertex tree. We now give the definition of the Ramsey number for general collections of hypergraphs.

\begin{definition}
Let $\h_1,\h_2,\dots,\h_k$ be nonempty collections of $r$-uniform hypergraphs. The Ramsey number $R^r_k(\h_1,\h_2,\dots,\h_k)$ is defined to be the minimum integer $N$ such that if the hyperedges of the complete $r$-uniform $N$-vertex hypergraph are colored with $k$ colors, then for some $1\le i \le k$, there is a monochromatic copy of a member of $\h_i$. If $k$ is clear by context, then we omit $k$ in this notation. If some of the collections $\h_i$ consist of a single hypergraph $\mathcal{G}$, then we write $\mathcal{G}$ in place of $\h_i = \{ \mathcal{G}\}$.
\end{definition}

Ramsey problems for a variety of hypergraphs and classes of hypergraphs have been considered (for a recent survey of such problems see~\cite{ramseysurvey}).  In this article, we will primarily be concerned with families of hypergraphs defined in a natural way from a given graph $G$ (or hypergraph $\h$).  In the case when $G$ is a path or a cycle, Berge~\cite{Berge} introduced a very general class of hypergraphs defined in terms of $G$.  In particular if $G=P_t$, the path with $t$ edges, then a Berge-$P_t$ is any hypergraph with $t$ hyperedges $e_1,e_2,\dots,e_t$ containing vertices $v_1,v_2,\dots,v_{t+1}$ such that $v_i,v_{i+1} \in e_i$ for all $1 \le i \le t$ (a Berge-cycle is defined analogously). 

The Ramsey problem for Berge-paths and cycles has received much attention.  Of particular interest is a result of Gy\'{a}rf\'{a}s and S\'{a}rk\"{o}zy~\cite{Gyafas_bergecycle} showing that the 3-color Ramsey number of a 3-uniform Berge-cycle of length $n$ is asymptotic to $\frac{5n}{4}$ (the 2-color case was settled exactly in~\cite{2colorcycles}).

The general definition of a Berge-$G$ for an arbitrary graph $G$ was introduced by Gerbner and Palmer in~\cite{gerbner2017extremal}.  Since their publication, the Tur\'an problem for Berge-$G$-free hypergraphs has been investigated heavily (see, for example~\cite{AnsteeBerge},~\cite{Tait}~and~\cite{thres}). Complete graphs were considered by several authors in~\cite{gyori},~\cite{iran},~\cite{newgyarfas},~\cite{generallemmas} and~\cite{gerbner}. However, the analogous Ramsey problem has not yet been investigated beyond the special cases of paths and cycles.

We will recall the definition of the set of Berge-copies of a graph $G$. In fact, we will give a more general definition in which rather than starting with a graph $G$ we may start with any uniform hypergraph.  



\begin{definition}
Let $\h=(V,\E)$ be a $k$-vertex $s$-uniform hypergraph.  Then given an integer $r \ge s$, $B\h$ (the set of Berge-copies of $\h$) is defined to be the set of $r$-uniform hypergraphs $\h'=(W,\F)$  such that there exist $U \subseteq W$ and bijections $\phi:V\to U$, $\psi:\E \to \F$ such that for all $e=\{u_1,u_2,\dots,u_s\}\in \E$,  $\{\phi(u_1),\phi(u_2),\dots,\phi(u_s)\} \subseteq \psi(e)$.  In this case, we call $U$ the \textit{core} of $\cH'$.
\end{definition}

\begin{remark}
For simplicity, we will often (when it cannot lead to confusion) say that a hypergraph is a $B\h$ to mean it is an element of $B\h$. For example we may, in a colored hypergraph, say that a certain hypergraph is a red $BK_t$, meaning that it is an element of the set $BK_t$ which is red.  Similar terminology will be used with respect to the other structures which we define later.  
\end{remark}

One of the main topics of the present paper is determining the Ramsey number of the set of Berge-copies of a hypergraph (mainly in the graph case). We show that the $2$-color Ramsey number of $BK_t$ versus $BK_s$ is linear. 
In particular, we prove the following theorem:

\begin{theorem}\label{berge-Ramsey}

\begin{displaymath}
R^3(BK_s, BK_t) = \begin{cases}
			t +s -1 &  \textrm{if $s=t=2$, $s=t=3$ or $\{s,t\}= \{2,3\}$ or $\{s,t\}= \{2,4\}$} ,\\
            t+s-2 & \textrm{if $s = 2$, $t\geq 5$, or $s=3$, $t\geq 4$ or $s=t=4$}, \\
            t+s-3 & \textrm{if $s \geq 4$ and $t\geq 5$.}
				  \end{cases}
\end{displaymath}            
\end{theorem}

For higher uniformity, we will show the following theorem.
\begin{theorem}\label{berge:4-uniform}
\begin{displaymath}
R^4(BK_t, BK_t) = \begin{cases}
			t +2 &  \textrm{if $2\leq t \leq 5$},\\
            t+1 & \textrm{Otherwise}.
				  \end{cases}
\end{displaymath}
\end{theorem}

Moreover, for general uniformity $k$ we prove
\begin{theorem}\label{berge:5-uniform}
For $k\geq 5$ and $t \ge t_0(k)$ (for $k=5$, $t_0 =23$ suffices),
\begin{displaymath}
R^k(BK_t, BK_t) = t.
\end{displaymath}
\end{theorem}

\begin{remark}
We remark that a similar direction (but with mostly non-overlapping results) has been pursued by two other groups independently~\cite{axenovich,othergroup}. In particular,~\cite{axenovich} is primarily concerned with non-uniform hypergraphs whereas we focus solely on the uniform case.  
\end{remark}

In addition to Berge-hypergraphs, we consider a variety of related structures.  First, we discuss a more restrictive class of hypergraphs defined from a given hypergraph $\h$.


\begin{definition}
Let $\h = (V,\E)$ be a $k$-vertex $s$-uniform hypergraph and let $S \subset V$.  The trace of $\h$ on $S$, denoted $\Tr(\h,S)$, is the hypergraph with vertex set $S$ and hyperedge set $\{h \cap S:h\in \E\}$.
Then, given $r \ge s$, $T \h$ is defined to be the set of $r$-uniform hypergraphs $\{\h':\Tr(\h',V(\h)) = \h\}$. 
For each such element $\h' \in T\h$, we refer to $V(\h)$ as the core of $\h'$.
\end{definition}

This notion originates from the idea of shattering sets and the Sauer-Shelah lemma~\cite{SS1,SS2,SS3}. This lemma provides an upper bound on the size of an $n$-vertex (non-uniform) hypergraph avoiding $\Tr(\h,S) = 2^S$ for all $k$-vertex sets $S$. Frankl and Pach~\cite{FranklPach} investigated the same problem with the restriction that the hypergraph is $r$-uniform. In the case when $\h$ is a (graph) cycle, $T\h$ was studied under the name weak $\beta$-cycle~\cite{weakbeta}. In the case of complete graphs, bounds were obtained by Mubayi and Zhao in~\cite{MubayiZhao}.  For a survey on extremal problems for traces see~\cite{tracesurvey}.

We now turn our attention to an even more restrictive notion called the expansion of a hypergraph.


\begin{definition}
Let $\h=(V,\E)$ be an $s$-uniform hypergraph.  The $r$-expansion $H \h$, for $r\ge s$, is defined to be the $r$-uniform hypergraph formed by adding $r-s$ distinct new vertices to every hyperedge in $\h$.  Precisely, for each hyperedge $e \in \E$, let $U_e = \{u_{e,1},u_{e,2},\dots,u_{e,r-s}\}$, and define $H \h = (V \cup (\cup_{e \in \E} U_e),\F)$ where $\F = \{e \cup U_e:e\in E\}$. We call $V$ the core of $\h$ and $V(\h) \setminus V$, the set of expansion vertices.  
\end{definition}
If $\h$ is a cycle we recover the well-known notion of linear cycle.  Ramsey and Tur\'an problems for linear cycles have been investigated intensely (see, for example~\cite{lincyc}).  The Tur\'an problem when $\h$ is a complete graph was investigated in~\cite{mubayi} and~\cite{Pik}. See~\cite{expansions} for a detailed survey of Tur\'an problems on expansions. In this article, we investigate the $2$-color Ramsey number of the 3-expansion of complete graphs $K_t$. By definition, a $3$-expansion of complete $K_t$ has $\binom{t}{2}+t$ vertices. Thus $R^3(HK_{t}, HK_{t}) \geq \binom{t}{2}+t$. We prove in the following theorem yielding a cubic upper bound on $R^3(HK_{t}, HK_{s})$.
\begin{theorem}\label{th:expansion-upper}
For $t, s\geq 2$, we have
\begin{displaymath}
R^3(HK_{t}, HK_{s}) \leq 2st(s+t).
\end{displaymath}
\end{theorem}

\begin{remark}
Suppose $t \ge s$, as a lower bound we can take a blue clique on $t + \binom{t}{2} -1$ vertices.  However, there is still a gap in the order of magnitudes of quadratic versus cubic.  
\end{remark}

In \cite{conlon2015hedgehogs}, Conlon, Fox and R{\"o}dl proved the same bound for diagonal Ramsey numbers. They showed that $R^3(HK_{t}, HK_{t}) \leq 4t^3$. This bound was latter improved by Fox and Li in~\cite{fox2019ramsey} where it was shown that $R^3(HK_{t}, HK_{t}) = O(t^2 \ln t)$.

Next we consider another way a hypergraph can be defined from another arbitrary hypergraph called a suspension~\cite{suspension} (or earlier enlargement~\cite{enlargement}).
Conlon, Fox and Sudakov considered the Ramsey numbers of the 3-suspension of a graph versus a 3-uniform clique in a short section of \cite{conlon2010hypergraph}. 

\begin{definition}
Let $\h=(V,\E)$ be an $s$-uniform hypergraph.  The $r$-suspension $S\h$, for $r \ge s$, is defined to be the hypergraph formed by adding a single fixed set of $r-s$ distinct new vertices to every edge in $\h$.  Precisely, let $U = \{u_1,u_2,\dots,u_{r-s}\}$, and define $S \h = (V \cup U, \F)$ where $\F = \{e \cup U:e \in E\}$. We call $V$ the core of $S\h$ and $U$ the set of suspension vertices.
\end{definition}

For suspensions of hypergraphs, we are only able to obtain Ramsey-type bounds using standard Ramsey number techniques. In particular, we show the following.
\begin{theorem}\label{th:suspension}
For $r\geq 3$, we have 
$$(1+o(1))\frac{\sqrt{2}}{e} t \sqrt{2}^t < R^r(SK_t, SK_t) \leq R^2(K_t, K_t)+ (r-2).$$
\end{theorem}

Finally, we discuss a a class of hypergraphs defined from a graph which is larger than the class defined by a Berge-hypergraph.

\begin{definition}
The $2$-shadow of a hypergraph $\h = (V,\E)$, denoted $\partial_2(\h)$, is the graph $G = (V,E)$ where $E = \{\{x,y\}:\{x,y\}\subseteq e\in \E\}$. Given a graph $G = (V,E)$, define $\partial G$ to be the set of hypergraphs $\{\h: E(G) \subseteq \partial_2(\h)\}$.
\end{definition}

In~\cite{mubayi}, Mubayi determined the Tur\'an number of $\partial K_t$ in all uniformities. In this paper, we prove the following. 

\begin{theorem}\label{thm:2-shadow}
We have
\begin{enumerate}[label=\rm{(\arabic*)}]
\item $R^3(\partial K_2, \partial K_2) = 3$.
\item $R^3(\partial K_2, \partial K_s) = s$ for $s\geq 3$.
\item $R^3(\partial K_t, \partial K_s) = t+s-3$ for $s,t\geq 3$.
\item $R^r(\partial K_t, \partial K_s) = \max\{s,t\}$ for $r \ge 4$ and $s,t \ge r$.
\end{enumerate}
\end{theorem}

\begin{remark} 
Observe that for any graph $G$, we have $\{HG,SG\} \subset TG \subset BG \subset \partial G$.  
\end{remark}

\noindent\textbf{Organization} 
The organization of the paper is as follows: In Section~\ref{sc:Berge}, we give the proof of Theorems~\ref{berge-Ramsey},~\ref{berge:4-uniform} and \ref{berge:5-uniform}. In Section~\ref{2shadow}, we give the proof of Theorem~\ref{thm:2-shadow}.  In Section~\ref{sc:trace}, we show some results on the Ramsey number of trace-cliques. In Section~\ref{sc:exp-susp}, we give the proof of Theorems~\ref{th:expansion-upper} and~\ref{th:suspension}.

\vspace{0.5cm}

\section{Ramsey number of Berge-hypergraphs}\label{sc:Berge}

To avoid tedious case analysis, some of the small cases are verified by computer. The code is available at \url{https://github.com/wzy3210/berge_Ramsey}. We list below the results verified by the computer.

\begin{proposition}\label{prop:comp} 
We have
\begin{enumerate}[label=\rm{(\arabic*)}]
    \item $R^3(BK_3, BK_4)= 5$.
    \item $R^3(BK_4, BK_5) = 6$.
    \item $R^4(BK_t,BK_t) \leq t+2$ for $2\leq t\leq 5$.
    \item $R^4(BK_6,BK_6) \leq 7$.
\end{enumerate}
\end{proposition}

\subsection{Proof of Theorem~\ref{berge-Ramsey}}

Recall that the number $R^3(BK_s, BK_t)$ is the smallest number $N$ such that any $2$-edge-colored complete  $3$-uniform hypergraph (with colors blue and red) on $n\geq N$ vertices either contains a blue Berge $K_s$ or a red Berge $K_t$.
In this subsection, we will show that
\begin{displaymath}
R^3(BK_s, BK_t) = \begin{cases}
			t +s -1 &  \textrm{if $s=t=2$, $s=t=3$ or $\{s,t\}= \{2,3\}$ or $\{s,t\}= \{2,4\}$} ,\\
            t+s-2 & \textrm{if $s = 2, t\geq s+3$, or $s=3, t\geq s+1$ or $s=t=4$}, \\
            t+s-3 & \textrm{if $s \geq 4$ and $t\geq 5$.}
				  \end{cases}
\end{displaymath}                  

Let us first deal with the cases when one of $s$ or $t$ is small. In particular, we prove them in the following proposition.

\begin{proposition}\label{prop:small-case}
We have 
\begin{enumerate}[label=\rm{(\arabic*)}]
\item\label{sm:1} $R^3(BK_2,BK_2) = 3$.
\item\label{sm:2} $R^3(BK_2,BK_3) = 4$.
\item\label{sm:3} $R^3(BK_3,BK_3) = 5$.
\item\label{sm:2-4} $R^3(BK_2,BK_4) = 5$.
\item\label{sm:4} $R^3(BK_4,BK_4) = 6$.
\item\label{sm:5} $R^3(BK_2,BK_t) = t$ when $t \geq 5$.
\item\label{sm:6} $R^3(BK_3,BK_t) = t+1$ when $t\geq 4$.
\end{enumerate}
\end{proposition}
\begin{proof}
\ref{sm:1} is trivial since any non-trivial edge-colored $3$-uniform hypergraph contains at least $3$ vertices and any edge is a $BK_2$. 
For~\ref{sm:2}, $R^3(BK_2, BK_3) > 3$ since a single red edge is a complete $K_3^{(3)}$ and is not a red $BK_3$. For the upper bound, suppose we have an edge-colored $K^{(3)}_4$. If it has a blue edge, we get a blue $BK_2$. Otherwise all of the $4$ edges are red, in which case we have a red $BK_3$. 
Similar reasoning gives~\ref{sm:2-4} and~\ref{sm:5}.
For~\ref{sm:3}, $R^3(BK_3,BK_3) > 4$ since an edge-colored $K^{(3)}_4$ with two red and two blue edges does not have a monochromatic $BK_3$. Similar reasoning gives the lower bound of~\ref{sm:4}. The upper bounds of~\ref{sm:3} and~\ref{sm:4} follow from Lemma~\ref{berge:induction}. 
For~\ref{sm:6}, we first show that $R^3(BK_3, BK_t) > t$. Let $\cH$ be an edge-color $K^{(3)}_t$ with two special vertices $v_1, v_2$ such that any hyperedge containing both $v_1, v_2$ is blue and all other hyperedges are colored red. Observe that any blue Berge clique or red Berge clique cannot contain both $v_1$ and $v_2$. Therefore, there is no blue $BK_3$ or red $BK_t$ in $\cH$.  For the upper bound, it is checked by computer that $R^3(BK_3, BK_4)= 5$ and the bound $R^3(BK_3, BK_t) \leq  t+1$ ($t\geq 5$) follows from Lemma~\ref{berge:induction}, which will be proven later.
\end{proof}

Next we show the lower bound in the following proposition.
\begin{proposition}\label{berge:lower}
Suppose  $s,t\geq 3$. We then have 
\begin{displaymath}
R^3(BK_t, BK_s) \geq t+s-3.
\end{displaymath}
\end{proposition}
\begin{proof}
We will construct a $2$-edge-colored complete $3$-uniform hypergraph $\cH$ on $t+s-4$ vertices without a blue $BK_t$ and red $BK_s$. Let $V(\cH) = A\sqcup B$ where $\abs{A} = t-2$ and $\abs{B} = s-2$. For all $a,a' \in A$, $b\in B$, color the hyperedge $\{a,a',b\}$ blue. For all $a \in A$, $b,b'\in B$, color the hyperedge $\{a,b,b'\}$ red. Moreover, color all triples in $A$ blue and all triples in $B$ red. Observe that any blue Berge clique contains at most one vertex from $B$ and any red Berge clique contains at most one vertex from $A$. It follows that $\cH$ does not contain a blue $BK_t$ or a red $BK_s$. Hence $R^3(BK_t, BK_s) \geq t+s-3$.
\end{proof}

Before we present the proof of Theorem~\ref{berge-Ramsey}, we will prove the following lemma.
\begin{lemma}\label{berge:induction}
\label{berge}
Suppose $t,s \geq 3$. Then
\begin{displaymath}
R^3(BK_t,BK_{s}) \leq \max\{R^3(BK_{t-1},BK_{s}),R^3(BK_{t},BK_{s-1})\}+1.
\end{displaymath}
\end{lemma}

\begin{proof}
Without loss of generality, assume $t \geq s$. Let $\h$ be a 2-edge-colored complete 3-uniform hypergraph with vertex set $V$ of size at least $N:= \max\{R^3(BK_{t-1},BK_{s}),R^3(BK_t,BK_{s-1})\}+1$. 
We want to show that $\h$ contains either a blue $BK_t$ or a red $BK_s$ as a sub-hypergraph. 

Fix $v\in V$ and let $\h'$ be the hypergraph induced by $V':=V\backslash\{v\}$. Since $\abs{V'} \geq R^3(BK_{t-1},BK_{s})$, it follows by definition that $\h'$ contains a blue $BK_{t-1}$ or a red $BK_{s}$. If there is a red $BK_s$ we are done. Otherwise suppose we have a blue $BK_{t-1}$, with the vertex set $Y$ as its core. Now let us consider $G$, the blue trace of $v$ in $\h$, i.e., $G$ is a 2-edge-colored complete graph with vertex set $V'$ and there exists an edge $\{x,y\}$ in $G$ if and only if the hyperedge $\{x,y,v\}$ in $\h$ is colored blue.

\begin{claim}\label{cl:large-red-degree}
Either we can extend $Y$ using $v$ to obtain a blue $BK_t$ or there exists a vertex $u\in Y$ with $d_G(u) \leq 1$. Moreover if $d_G(u) = 1$ and $\{u,w\}$ is the only edge containing $u$, then $d_G(w) <  N-2.$
\end{claim}

\begin{proof}

Consider the incidence graph of $G$, i.e., the bipartite graph $I = Y \cup E(G)$ such that for every $u\in Y$, $e\in E(G)$, $u$ is incident to $e$ if and only if $u\in e$. 
Observe that $Y$ is the core of a blue $BK_{t-1}$ with none of its hyperedges containing $v$.
Therefore, by our definition of $G$ (the blue trace of $v$ in $\h$), if there is a matching of $Y$ in $I$, then we can obtain a blue $BK_t$ with $Y\cup \{v\}$ as its core.

Now assume $I$ does not contain a matching of $Y$. We first claim that there exists a vertex $u\in Y$ with $d_G(u) \leq 1$. Note that the degree of each $e \in E(G)$ is at most $2$. Thus, if $d_I(u) \geq 2$ for all $u\in Y$, then it follows that for every $S\subseteq Y$, $|N_I(S)| \geq |S|$, which gives us a matching on $Y$ by Hall's condition. Thus by contradiction, we have a vertex in $Y$ of degree at most $1$ in $G$.

Suppose now $d_G(u) = 1$ for some $u$ in $Y$ and $e = \{u,w\}$ is the unique edge containing $u$. We claim that $d_G(w) < N-2$. Suppose not, i.e., $d_G(w) \geq N-2$. This implies that $\{v,w,z\}$ is a blue edge for every $z \in V(\cH)\backslash \{v,w\}$. Moreover, by our lower bound in Proposition~\ref{prop:small-case} (when $s,t$ are small) and Proposition~\ref{berge:lower}, there exists another vertex $y \in V'\backslash Y$. It follows that we can extend $Y$ into the core of a blue $BK_t$ with the following embedding:
for each $z \in Y\backslash\{w\}$, embed $\{v,z\}$ to the hyperedge $\{v,z,w\}$. Then embed $\{v,w\}$ to $\{v,w,y\}$. Thus if we do not have a blue $BK_t$ with $Y\cup v$ as its core, then we have $d_G(w) < N-2$.
\end{proof}

This claim says that either there exists $u \in Y$ such that $\{v,u,x\}$ is red for every $x \in V'\backslash\{u\}$, or there exists $u,w\in V'$ such that $\{v,u,x\}$ is red for every $x\not= w$ and there exists $w_x$ such that $\{v,w,w_x\}$ is red. Note that the second case covers the first case by taking $w_x = u$. So it suffices to assume the second case.

Now since $N-1 \geq R^3(BK_{t},BK_{s-1})$, it follows that $\h'$ either contains a blue $BK_{t}$ or a red $BK_{s-1}$. We are done in the former case. Otherwise, suppose that $\h'$ contains a red $BK_{s-1}$. We will show that we can extend this $BK_{s-1}$ by adding the vertex $v$ into its core.
Let $X$ be the core of the Berge-$K_{s-1}$. Now for every $x\in X$ with $x\notin \{u,w\}$, we know that the edge $\{v,u,x \}$ is colored red. Hence we can embed $\{v,x\}$ into the red hyperedge $\{v,u,x\}$. It follows that  we have an embedding of the edges from $v$ to all but at most two vertices of $X$, namely $u,w$. In the case that $w \in X$, we can embed $\{v,w\}$ into the hyperedge $\{v,w,w_x\}$, which is red. Now if $u\notin X$, we are done. Otherwise, assume $u \in X$. Note that by the lower bounds in Proposition~\ref{prop:small-case} (when $s,t$ are small) and Proposition~\ref{berge:lower}
$|V'|=N-1\geq \max\{R^3(BK_{t-1}, BK_{s}), R^3(BK_t, BK_{s-1})\} \geq s+1.$ Hence it follows that there exists another vertex $y \in V(\cH')\backslash (X\cup\{w\})$. Note that by our choice of $u$, $\{v,u,y\}$ is red. Thus we can embed $\{v,u\}$ into $\{v,u,y\}$. The above embedding extends $X$ into the core of a red $BK_s$ and we are done.
\end{proof} 

\begin{lemma}\label{lm:4t-base-case}
$R^3(BK_4,BK_t) = t+1$ for $t \geq 5$.
\end{lemma}

\begin{proof}
We will proceed by induction on $t$. The base case that $R^3(BK_4, BK_5) = 6$ is verified by computer. Suppose now that Lemma~\ref{lm:4t-base-case} is true for all $5\leq t' < t$. Let $\h$ be a $2$-edge-colored complete 3-uniform hypergraph on $t+1$ vertices. Note that by Proposition~\ref{prop:small-case}, we have $R^3(BK_3, BK_{t}) = t+1$. Hence $\cH$ either contains a blue $BK_3$ or a red $BK_{t}$. If the latter happens, we are done. So suppose $\cH$ contains a blue $BK_3$, with the vertex set $Y$ as its core.
Note that $t+1 \geq 7$ and a Berge-triangle contains at most $6$ vertices. Hence there exists a vertex $v$ that is not used by any hyperedge in the blue $BK_3$. Similar to Lemma \ref{berge:induction}, let $G$ be the blue trace of $v$ in $\cH$. Again by Claim \ref{cl:large-red-degree}, either we can extend $Y$ using to $v$ to obtain a blue $BK_4$ or there exists a vertex $u \in Y$ with $d_G(u)\leq 1$. Moreover, if $d_G(u)=1$ and $\{u,w\}$ is the only edge containing $u$, then $d_G(w) < t-1$. In the former case, we are done. Otherwise, WLOG, assume that there exists a $u \in Y$ and $w \in V(\cH)\backslash\{v,u\}$ such that $\{v,u,x\}$ is red for every $x\neq w$ and there exists some vertex $w_x$ such that $\{v,w,w_x\}$ is red. By induction, $\cH[V(\cH)\backslash\{v\}]$ contains either a blue $BK_4$ or a red $BK_t$. In the former case, we are done. In the latter case, we can extend the red $BK_t$ to a red $BK_{t+1}$ in the same way as in Lemma~\ref{berge:induction}.
\end{proof}

Now this result together with Lemma~\ref{berge} allows us to show the following proposition.

\begin{proposition}\label{prop:berge-general}
$R^3(BK_t,BK_s) \leq t + s -3,$ for $t,s \geq 4$ and $\max\{s,t\}\geq 5$.
\end{proposition}

\begin{proof}
We already know this is true if one of $t$ or $s$ is 4, and so for $t,s \geq 5$ the result follows from induction on $t+s,$ using Lemma~\ref{berge}.
\end{proof}

\noindent Theorem~\ref{berge-Ramsey} follows from Proposition~\ref{prop:small-case},~\ref{berge:lower}~and~\ref{prop:berge-general}.

\subsection{Proof of Theorem~\ref{berge:4-uniform}}
In this section, for ease of reference, sometimes we use the notation $h \to e$ to denote that the hyperedge $h \in E(\cH)$ is mapped to the vertex pair $e \in E(G)$ when constructing the embedding of $E(G)$ in $E(\cH)$.

Let us first deal with Theorem~\ref{berge:4-uniform} for small values of $t$.

\begin{proposition}\label{prop:4-uniform-small}
For $2\leq t\leq 5$, $R^4(BK_t,BK_t)= t+2$.
\end{proposition}
\begin{proof}
For the lower bound, we use the fact that if $R^4(BK_t,BK_t) = n$, then $\binom{n}{4} \geq 2\binom{t}{2}-1$. For $2\leq t\leq 5$, this shows that $R^4(BK_t,BK_t) \geq t+2$. The upper bound that $R^4(BK_t,BK_t) \leq t+2$ for $2\leq t\leq 5$ is verified by computer.
\end{proof}

Now we want to show that $R^4(BK_t, BK_t) = t+1$ for all $t \geq 6$.
Again we start with the lower bound by showing the following proposition.
\begin{proposition}\label{4-uniform-lower}
$R^4(BK_t, BK_t) \geq t+1$ for all $t\geq 6$.
\end{proposition}
\begin{proof}
We want to construct a $2$-edge-coloring of a complete $4$-uniform hypergraph on $t$ vertices without a monochromatic $BK_t$. Let $\cH$ be a $K^{(4)}_{t}$ with two special vertices $v_1, v_2$. Any hyperedge containing both $v_1, v_2$ is colored blue. All other hyperedges are colored red. We claim that there is no monochromatic $BK_t$ in $\cH$. Indeed, there is no red $BK_t$ since only one of $v_1, v_2$ can be in any red $BK_t$. For blue $BK_t$, note that by our coloring there are only $\binom{t-2}{2}$ blue edges, which are fewer than the $\binom{t}{2}$ edges needed for $BK_t$.
\end{proof}

Now let us move on to the upper bound. 

\begin{lemma}\label{4-uniform-upper}
For $t\geq 6$, we have that 
$$R^4(BK_t, BK_t) \leq t+1.$$
\end{lemma}
\begin{proof}
We prove the lemma by inducting on $t$.
The base case that $R^4(BK_6,BK_6) \leq 7$ is verified by computer. Now assume that $t \geq 7$ and the lemma is true for all $t'<t$.

Let $\cH$ be a $2$-edge-colored complete $4$-uniform hypergraph on a vertex set $V$ of size $t+1$. For ease of reference, given a set of vertices $S$, let $d_b(S)$ and $d_r(S)$ denote the number of blue and red hyperedges containing $S$ as subset, respectively.

\begin{claim}\label{cl:almost-blue}
Suppose $\cH$ does not contain a monochromatic $BK_{t}$.
Let $v$ be a fixed vertex in $\cH$. If there is a monochromatic $BK_{t-1}$ (without loss of generality, assume it is blue) without using any hyperedge containing $v$, then there exists another vertex $u$ such that $d_b(\{v,u\})\leq 2$, i.e., all hyperedges containing both $v,u$ are red except for at most two.
\end{claim}
\begin{proof}
Let $\cH_b$ be the blue Berge-$K_{t-1}$ hypergraph not using any hyperedge containing $v$. Let $\{u_1, u_2, \ldots u_{t-1}\}$ be the core of $\cH_b$. Construct a bipartite graph $G= A \cup B$ where $A = \{u_1, \ldots, u_{t-1}\}$ and $B =\binom{V\setminus\{v\}}{3}$. For $u_i\in A$, $S \in B$, $u_i$ is adjacent to $S$ in $G$ if and only if $u_i \in S$ and $\{v\} \cup S$ is a blue edge in $\cH$.  Note that for every $S \in B$, $d_G(S) \leq 3$. Therefore, if $d_G(u_i) \geq 3$ for every $u_i \in A$, then there exists a matching of $A$ in $G$ by Hall's theorem, which implies that we can extend $\cH_b$ to a blue $BK_{t}$ by adding $v$ into the core of $\cH_b$. This contradicts our assumption that $\cH$ does not have a monochromatic $BK_{t}$, and the proof of Claim~\ref{cl:ex:induct} is complete.
\end{proof}

Now for every $v \in V$, there exists a monochromatic $BK_{t-1}$ in $\cH[V\backslash\{v\}]$ by induction. Hence by Claim~\ref{cl:almost-blue}, for every vertex $v$, there exists another $u$ in $V$, such that $d_c(\{v,u\})\geq \binom{t-1}{2}-2$ for some $c \in$ \{blue,red\}. We then call the pair $\{v,u\}$ a \textit{c couple} where $c \in$ \{blue,red\}. Moreover, call $\{a,b\}$ a `bad pair' of $\{v,u\}$ if the hyperedge $\{a,b,v,u\}$ is not in color $c$.

By Claim~\ref{cl:almost-blue}, every vertex is contained in a couple. It follows that we have at least $(t+1)/2 \geq 4$ couples so at least two of them are of the same color. Without loss of generality, let $\{v_1,u_1\}$ and $\{v_2,u_2\}$ be two red couples. Our goal is to obtain a red embedding of a $BK_t$ using mostly edges containing $\{v_1,u_1\}$ and $\{v_2,u_2\}$. We assume that $\{v_1,u_1\} \cap \{v_2,u_2\} = \emptyset$ and remark that the other case is similar and simpler.
Let $\{a_1,b_1\}, \{a_2,b_2\}$ be the two possible bad pairs of $\{v_1,v_2\}$. Let $\{c_1,d_1\}$, $\{c_2,d_2\}$ be two possible bad pairs of $\{v_2,u_2\}$. 
If $\{v_1, u_1\}$ has exactly two bad pairs, we can assume that for at least one of them (with loss of generality the pair $\{a_2, b_2\}$) there is a red edge $h$ containing it. Otherwise $\{a_1,b_1\}$ and $\{a_2,b_2\}$ are blue couples with no bad pairs and it is easy to find a blue $BK_t$ by only using the blue edges containing $\{a_1,b_1\}$ and $\{a_2,b_2\}$.

If $\{v_1, u_1\}$ has exactly one bad pair, let $\{a_1,b_1\}$ be that pair and pick $\{a_2, b_2\}$ arbitrarily. Note that $\{a_2, b_2\}$ is contained in some red edge $h$.
If $\{v_1,u_1\}$ has no bad pair, then pick  $\{a_1,b_1\}$ and $\{a_2,b_2\}$ arbitrarily.
Moreover, we assume that $\{v_1,u_1,v_2,u_2\}$ is a red edge and observe that otherwise constructing the embedding is easier.

Suppose $\{a_1,b_1\}$ and $\{a_2,b_2\}$ have a common vertex $u$. If $u \notin \{v_2, u_2\}$, relabel $a_1, b_1$ such that $a_1 = u$, and if $u \in \{v_2, u_2\}$ relabel $u_2, v_2,a_1, b_1$ such that $b_1 = u_2 = u$. Otherwise just relabel $a_1, b_1$ such that $a_1 \not \in \{v_2,u_2\}$.  Let $x_1,x_2,\dots,x_{t-4}$ be an enumeration of $V':= V \setminus \{v_1,v_2,u_1,u_2,a_1\}$. If $b_1 \not\in \{v_2,u_2\}$, assume $x_1 = b_1$. Othewise assume without loss of generality that $b_1=u_2$. We are going to construct the embedding in three phases:

\begin{description}
\item \textit{Phase 1:} Embed all vertex pairs in $V'$.

Consider the following embedding:
For $i,j\in\{1,\dots,t-4\}$, embed $\{x_i,x_j\}$ in $\{u_1,v_1,x_i,x_j\}$ if $i+j$ is odd otherwise in $\{u_2,v_2,x_i,x_j\}$.

We have a red $BK_{t-4}$ except possibly for at most three missing edges. Without loss of generality, let $\{x_{i_1},x_{j_1}\}$, $\{x_{i_2},x_{j_2}\}$, $\{x_{i_3},x_{j_3}\}$ be the three possible bad pairs where $i_1+j_1$ is odd and both $i_2+j_2$ and $i_3+j_3$ are even. If $\{x_{i_1}, x_{j_1}\}$ is indeed a bad pair of $\{v_1, u_1\}$, then it follows that $\{x_{i_1}, x_{j_1}\} = \{a_2, b_2\}$.
Then we can embed $\{x_{i_2},x_{j_2}\}$ in $\{v_1, u_1, x_{i_2},x_{j_2}\}$, embed $\{x_{i_3},x_{j_3}\}$ in $\{v_1, u_1, x_{i_3},x_{j_3}\}$ and embed $\{x_{i_1}, x_{j_1}\}$ in $h$. Otherwise, $\{x_{i_1}, x_{j_1}\}$ does not exist and the above embedding still works except when one of $\{x_{i_2},x_{j_2}\},\{x_{i_3},x_{j_3}\}$ is the pair $\{a_2, b_2\}$. We can then use $h$ to embed $\{a_2, b_2\}$.

\item \textit{Phase 2:} Embed all edges from $\{v_1, u_1, v_2, u_2\}$ to vertices in $V'$.

Consider the following embedding:
\begin{align*}
 \{v_1,u_1,a_1, x_i\} &\to  \{x_i,u_1\} \textrm{ for $i\neq 1$}.\\
 \{v_1,u_1,v_2, x_i\} &\to \{x_i,v_1\} \textrm{ for $i\neq 1$}. \\
 \{v_2,u_2,a_1, x_i\} &\to \{x_i,u_2\}.\\
 \{v_1,v_2,u_2 x_i\} &\to \{x_i,v_2\}.
\end{align*}

Note that  $x_1$ can only be contained in one bad pair otherwise we would have picked $x_1$ to be $a_1$. Hence among the three edges $\{v_1,u_1,x_1,v_2\}$, $\{v_1,u_1,x_1,u_2\}$, $\{v_1,u_1,a_1,x_1\}$, at least two of them are red.
Embed $\{x_1, v_1\}$, $\{x_1, u_1\}$ into those two red edges. If all three are red, do not use $\{v_1,u_1,u_2,x_1\}$ in this part of the embedding.

Now let us analyze the potential bad cases. There are at most $3$ of these edges in Phase 2 that are not red. 

If $\{u_1,v_1,a_1,x_i,\}, i\not = 1$ is blue, then use the edge $\{v_1,u_1,u_2,x_i\}$ to embed $\{u_1,x_i\}$.

If $\{v_1,u_1,v_2,x_i\},i\not=1$ is blue, then use the edge $\{v_1,u_1,u_2,x_i\}$ to embed $\{v_1,x_i\}$.

If there are two different indexes $i,j$ such that $h_1 \in \{\{v_2,u_2,a_1,x_i\}, \{v_1,v_2,u_2,x_i\}\}$  and $h_2 \in \{\{v_2,u_2,a_1,x_j\},\{v_1,v_2,u_2,x_j\}\}$ are blue, then replace $h_1$ with $\{u_1, v_2, u_2,x_i\}$ and replace $h_2$ with $\{u_1, v_2, u_2, x_j\}$. The same embedding works if there is only one bad pair of $\{v_2, u_2\}$ in this phase.

If for some $i$  both edges $\{v_1,v_2,u_2,x_i\}, \{v_2,u_2,a_1,x_i\}$ are blue, 
then it follows that the edge $\{v_2,u_2,x_i,y\}$ is red for every vertex $y$, with $y\notin \{v_1,a_1,v_2,u_2,x_i\}$. 
Consider the set of edges $E_i = \{\{v_2,u_2,x_i,y\}: y\notin \{v_1,v_2,u_2,a_1,x_i\} \}$. Note that $|E_i| = t-4$. 
In Phase $1$, at most $\ceil{(t-6)/2}$ edges in $E_i$ are used except when $t$ is even and $i$ is odd, in which case $\floor{(t-6)/2}$ edges in $E_i$ are used. If $t$ is even and $i$ is odd, we have at least $t-4-\floor{(t-6)/2} \geq 3$ edges in $E_i$ still available. 
In other cases, we have at least $t-4-\ceil{(t-6)/2} \geq 2$ edges in $E_i$ still available.
Either there exist two edges in $E_i$ that can be used to embed $\{v_2,x_i\}$ and $\{u_2,x_i\}$, or in Phase $1$ there exists some $j$ such that $\{v_1,u_1,x_i,x_j\}$ is blue and $\{v_2, u_2, x_i,x_j\}$ is used to embed $\{x_i, x_j\}$.
In this case, there exists some $k \in \{1,\dots t-4\}\backslash \{i\}$ such that $i+k$ is even and $\{v_1, u_1,x_i, x_k\}$ is red. Embed $\{x_i, x_k\}$ into $\{v_1, u_1, x_i, x_k\}$.
It follows that we again have two available red edges containing $x_i, v_2, u_2$ to embed $\{ v_2,x_i\}$, $\{u_2,x_i\}$.

\item \textit{Phase 3:} Embed the edges in $\displaystyle\binom{\{u_1,v_1,u_2,v_2\}}{2}$.

If the edge $\{u_1,v_1,v_2,a_1\}$ is red, then use it to embed $\{v_1,v_2\}$. 
Otherwise we know that $\{v_2,a_1\}$ and $\{u_2,a_1\}$ are the two bad pairs of $\{v_1,u_1\}$. 
It follows that the edge $\{v_1,u_1,u_2,x_1\}$ is still available and the edge $\{v_1,u_1,v_2,x_1\}$ was used to embed $x_1$ with one of $v_1$ or $u_1$ (without loss of generality, assume $v_1$).
In this case, embed $\{v_1,x_1\}$ in $\{v_1,u_1,u_2,x_1\}$ instead and use the edge $\{v_1,u_1,v_2,x_1\}$ to embed $\{v_1,v_2\}$. 
Now we will embed $\{v_1,u_2\}$ and $\{u_1,u_2\}$.
Let $E_{u_2} = \{\{v_1,u_1,u_2,y\}: y \notin \{v_1,u_1, v_2,  u_2\}\}$. Note that $|E_{u_2}| = t-3$ and at most $2$ edges in $E_{u_2}$ are blue. 
Hence at least $(t-3)-2\geq 2$ of the edges in $E_{u_2}$ are red.
For each red edge in $E_{u_2}$, if it was used, it was because there exists some bad pair of $\{v_1,u_1\}$ which did not use $u_2$.
That in turn implies that there are still at least $2$ edges in $E_{u_2}$ that are red and available.
Hence we can embed $\{v_1,u_2\}$ and $\{u_1,u_2\}$ into these two edges. 
Similarly we can find an edge of the form $\{v_2,u_1,u_2,y\}$ to embed $\{u_1,v_2\}$.

Finally, by counting the edges used, it is easy to check that there are still red edges of the form $\{v_1,u_1,x,y\}$ and $\{v_2,u_2,x,y\}$ available to embed both $\{v_1, u_1\}$ and $\{v_2,u_2\}$, since each pair is in at least $\binom{t-1}{2}-2$ red edges. \qedhere
\end{description}
\end{proof}

In the case of cliques of different sizes we have the following bounds which are trivial from Theorem~\ref{berge:4-uniform}.

\begin{proposition}
Suppose $t \ge s \geq 2$ and $t \ge 6$, then
\begin{displaymath}
t \le R^4(BK_t,BK_s) \le t+1.
\end{displaymath}
\end{proposition}
\begin{proof}
The construction is trivial, we just take a clique on $t-1$ vertices.  The upper bound follows since $s \le t$ implies $R^4(BK_t,BK_s) \le R^4(BK_t,BK_t)$.
\end{proof}
For $s=t-1$ we obtain the same bound as the case $s=t$.
\begin{proposition}
$R^4(BK_t,BK_{t-1}) = t+1$ for $t\geq 6$.
\end{proposition}
\begin{proof}
The same construction works as the $R^4(BK_t,BK_{t})$ case, and the upper bound follows from $R^4(BK_t,BK_{t-1})\le R^4(BK_t,BK_{t})$.
\end{proof}


\begin{theorem}
Assume $2 \le s \le t-2,$ and $t\geq 34$, then $R^4(BK_t,BK_{s}) = t.$
\end{theorem}
\begin{proof}

In a red-blue coloring of a hypergraph $\h$, given a pair of vertices $\{v,u\}$, we define its blue degree to be $d_B(\{v,u\}) = \abs{h\in E(\h):\{v,u\}\subseteq h \mbox{ and $h$ is blue}\}}$.  The red degree $d_R(\{v,u\})$ is defined analogously. Let 
\begin{displaymath}
\delta_B^2 = \min_{\{v,u\}\in \binom{V(\h)}{2}}d_B(\{v,u\}),
\end{displaymath}
and define $\delta_R^2$ similarly.  


Call $\{v,u\}$ a $c$ couple, $c\in \{blue,red\}$, if all but at most 5 of the hyperedges $\{v,u,x,y\}$ are $c$ colored, and also call a pair $\{x,y\}$ a bad pair of the $c$ couple $\{v,u\}$ if the hyperedge $\{v,u,x,y\}$ is not colored $c$.

Note that if $\delta_B^2=0$ then we can find a pair $\{v,u\}$ such that $\{v,u,x,y\}$ is red for all $x,y$, and therefore there is a red $BK_{t-2}$. So we can assume $\delta_B^2 \geq 1$.


\begin{claim} \label{cl3} Suppose there are two blue couples, then either we can find a blue $BK_t$ or we can find two red couples such that each have at most 4 bad pairs.\end{claim}


\begin{proof}
Assume we have two disjoint blue couples $\{u_1,v_1\}$ and $\{u_2,v_2\}$, the case where these pairs are not disjoint is similar and simpler,  and enumerate the other $t-4$ vertices as $x_1,x_2,\dots,x_{t-4}.$  
Now let us do a preliminary embedding, for $i,j\in [t-4]$ use $\{u_1,v_1,x_i,x_j\}$ to embed  $\{x_i,x_j\}$ when $i+j$ is odd and $\{u_2,v_2,x_i,x_j\}$ otherwise. 
If $i+j$ is odd and in this part of the embedding we used a red edge $\{u_1,v_1,x_i,x_j\}$ to embed $\{x_i,x_j\}$, but the edge $\{u_2,v_2,x_i,x_j\}$ is blue, then use the edge $\{u_2,v_2,x_i,x_j\}$ instead. If $i+j$ is even and in this part of the embedding we used a red edge $\{u_2,v_2,x_i,x_j\}$ to embed $\{x_i,x_j\}$, but the edge $\{u_1,v_1,x_i,x_j\}$ is blue, then use the edge $\{u_1,v_1,x_i,x_j\}$ instead.  Let us call such a change to the embedding a swap.  If both edges $\{u_1,v_1,x_i,x_j\}$ and $\{u_2,v_2,x_i,x_j\}$ are red or blue, then we do not change anything.

Note that at this point we have embedded a $BK_{t-4}$ such that every edge is blue except at most five edges, in particular the possible pairs which are simultaneously bad pairs of $\{u_1,v_1\}$ and $\{u_2,v_2\}$.

Let $e_1,e_2,\dots,e_k$ be these common bad pairs, $k \leq 5$. We begin with a simple observation which we will use again later.
\begin{observation} \label{finalobs}
If $k \le 1$ we could complete the embedding in such a way that each pair is contained in at least 1 blue edge.
\end{observation}
If $k \geq 2$ and all but at most one $e_i$ is in at least 5 blue edges, then we can greedily embed the edges, starting from the one that is in less than 5 blue edges, since each is in at least one unused blue edge. So we can either find two  of the $e_i$ which are in at most 4 blue edges and the claim is proven or we complete the embedding of a blue $BK_{t-4}$, and if that is the case we will see we can complete this embedding to a blue $BK_{t}.$



Since for any fixed $i$, there are at most $\lceil \frac{t-4}{2} \rceil$ indices $j$ such that $i+j$ is odd and also $x_i$ can be in at most 10 bad pairs of $\{u_1,v_1\}$ or $\{u_2,v_2\}$, it follows that for every $i\in [t-4]$ there are at least $t-5 - \lceil \frac{t-4}{2} \rceil - 10 \geq 4$ values of $j \in [t-4]$ not used in the previous steps of the embedding such that the edge $\{u_1,v_1,x_i,x_j\}$ is blue. Then again by Hall's Theorem in the incidence graph with components $X=\{\{x_i,v_2\}:i\in [t-4]\} \cup \{\{x_i,u_2\}:i \in [t-4]\}$ and $Y$ the set of blue edges in $\{\{x_i,x_j,u_2,v_2\}: 1\leq i < j \leq t-4\}$, we can find an embedding of the edges $\{x_i,v_2\}$ and $\{x_i,u_2\}$ for $i \in [t-4]$, and similarly we can find an embedding of the edges $\{x_i,v_1\}$ and $\{x_i,u_1\}$ for $i \in [t-4]$.

We have not yet used the hyperedges of the form $\{v_1,u_1,v_2,y\}$; there are at least $t-8 \geq 26$ of these which are blue, and we can use them to embed $\{v_1,u_1\}, \{v_1,v_2\}$ and $\{u_1,v_2\}$. Similarly we can embed $\{v_2,u_2\},\{u_1,u_2\}$ and $\{u_1,u_2\}$.
Therefore either we can complete the matching or we find two pairs $e_1,e_2$ which are red couples, with at most 4 bad pairs. This completes the proof of Claim~\ref{cl3}.
\end{proof}

\begin{claim}\label{cl4}
Suppose there are two red couples such that at least one has at most 4 bad pairs, then either we can find a red $BK_{t-2}$ or we can find two blue couples such that each have at most 1 bad pair.\end{claim}

\begin{proof}
Again we will assume the red couples are disjoint. Let $\{u_1,v_1\}$ and $\{u_2,v_2\}$ be couples such that $\{u_1,v_1\}$ has at most 4 bad pairs, and let $\{a_1,b_1\},\{a_2,b_2\},\{a_3,b_3\},\{a_4,b_4\}$ be the bad pairs of $\{u_1,v_1\}$.  Suppose these pairs are arranged by their red degree in increasing order. Now let $x_1,x_2,\dots,x_{t-6}$ be an enumeration of the set $V'= V\backslash \{v_1,v_2,u_1,u_2,a_1,a_2\}$. Let us consider the following embedding which is similar to the one used in the previous claim:
For $i,j\in [t-6]$ use $\{u_1,v_1,x_i,x_j\}$ to embed  $\{x_i,x_j\}$ when $i+j$ is odd and $\{u_2,v_2,x_i,x_j\}$ otherwise. 
Similarly as in Claim~\ref{cl3}, if we encounter a bad pair of one couple but not the other, then we can change the embedding to use more red edges, and at the end we have an embedding of a $BK_{t-6}$ with almost every edge red, the only possible exceptions are the common bad pairs of $\{u_1,v_1\}$ and $\{u_2,v_2\}$ in $V'$. Hence here we have at most two ($\{a_3,b_3\}$ and $\{a_4,b_4\}$). If the red degree of these edges is at least 2, then we can greedily embed these two in these pairs to complete a red clique on $V'$. Otherwise one of these, and by the ordering also $\{a_1,b_1\}$ and $\{a_2,b_2\}$, will be in at most 1 red pair.

Similarly as in the proof of Claim~\ref{cl3}, we use Hall's theorem to embed $\{x_i,v_2\}$, $\{x_i,u_2\}$, $\{x_i,v_1\}$ and $\{x_i,u_1\}$ for $i \in [t-6]$ (here the number  $t-5 - \lceil \frac{t-4}{2} \rceil - 10$ is replaced by $t-7 - \lceil \frac{t-6}{2} \rceil - 8$, which is at least 5).


Since $\{v_1,u_1,v_2,y\}$ is red for at least $t-7\geq 29$, and these hyperedges have not been used yet, it follows that we have enough hyperedges to embed $\{v_1,u_1\},\{v_1,v_2\}$ and $\{u_1,v_2\}$ and similarly we can embed $\{v_2,u_2\},\{v_1,u_2\}$ and $\{u_1,u_2\}$.
\end{proof}

Note that if there is at most one blue couple, say $\{v,u\}$, we may put $V' = V\backslash\{u\}$ and for every pair $x,y \in V'$ the red degree of $\{x,y\}$ is at least 6. Then by Hall's Theorem, we can find a red $BK_{t-1}$. So we can assume there are at least two blue couples. Thus, by Claim~\ref{cl3} either we find a blue $BK_{t}$ or we have two red couples such that at least one has at most 4 bad pairs, the conditions of  Claim~\ref{cl4}. From here we either find a red $BK_{t-2}$ or satisfy conditions stronger than those of Claim~\ref{cl3}. In this case, there is at most one shared bad pair and so we would be able to find a blue $BK_{t}$ by Observation~\ref{finalobs}. 
\end{proof}

\begin{remark}
Instead of using Hall's Theorem in the second part of the embedding on the previous claims, if we use a more complicated case analysis the constraint $t \ge 34$ can be relaxed somewhat, but we elected not to in order to make the proof easier to follow.
\end{remark}

\subsection{Proof of Theorem \ref{berge:5-uniform}}

In this short section, we will show that $R^k(BK_t,BK_t) = t$ when t is sufficiently large. 

\begin{claim}\label{cl:min-red-deg}
If for all $v,u \in V$, there are at least $\binom{k}{2}$ red distinct hyperedges containing both $v$ and $u$, then $\h$ contains a red $BK_t$. 
 \end{claim}
\begin{proof}
Consider the bipartite graph $G$ with vertex set $V(G) = A \sqcup B$, where $A = \binom{V(\h)}{2}$ and $B$ is the set of all hyperedges of $\h$. For $a \in A$, $h \in B$, $a$ is adjacent to $h$ in $G$ if and only if $a \subset h$ and $h$ is colored red in $\h$. Note that for every $h\in B$, $d_G(h) \leq \binom{k}{2}$.
Hence, if for all $\{v,u\} \in A$, $d_G(\{v,u\}) \geq \binom{k}{2}$, then by Hall's theorem we have a matching of $A$ in $G$, which implies the existence of a red $BK_t$ in $\h$.
\end{proof}

\begin{claim}\label{cl:k-uniform}
If $\binom{t-4}{k-4} \geq 2 \binom{k}{2} -1$, then $R^k(BK_t,BK_t) \leq t$.  
\end{claim}

\begin{proof}

If the condition in Claim \ref{cl:min-red-deg} does not hold, then there exist two vertices $v, u \in V(\h)$ such that all but at most $\binom{k}{2}-1$ hyperedges containing both $v$ and $u$ are blue. We claim that there exists a copy of a blue $BK_t$ in $\h$ using only blue hyperedges containing both $v$ and $u$.
Consider again the bipartite graph $G$ with vertex set $V(G)= A \sqcup B$, where $A = \binom{V(\h)}{2}$ and $B$ is the set of blue hyperedges of $\h$ containing both $v$ and $u$. Note that for every $a\in A$ there are at least $\binom{t-4}{k-4} - \binom {k}{2}+1 \geq \binom{k}{2}$ blue hyperedges containing $a$, and again by Hall's theorem we have a blue~$BK_t$.
\end{proof}

Using Claim \ref{cl:k-uniform}, we show that $R^k(BK_t, BK_t) = t$ when $k\geq 5$ and $t$ sufficiently large. We did not make an attempt to find the best possible constant.
\begin{corollary} We have

\begin{enumerate}[label=\rm{(\arabic*)}]
\item $R^5(BK_t,BK_t) = t$ when $t\geq 23$.
\item $R^6(BK_t,BK_t) = t$ when $t\geq 13$.
\item $R^7(BK_t,BK_t) = t$ when $t\geq 12$.
\item $R^{k}(BK_t,BK_t) = t$ when $k \in \{8,9,10\}$ and $t \geq k+4$.
\item $R^k(BK_t,BK_t) = t$ when $k \geq 11$ and $t \geq k+3$.
\end{enumerate}
\end{corollary}

\begin{remark}
Note that for $k\geq 11$, this result is sharp since for $t = k+2$ we have that $\binom{t}{r}\leq 2\binom{t}{2} -2.$ Hence $R^k(BK_t,BK_t) \geq r+3.$
\end{remark}

\subsection{Superlinear lower bounds for sufficiently many colors} \label{superlinear}
In this subsection we show that for all uniformities and for sufficiently many colors, the Ramsey number for a Berge $t$-clique is superlinear. We start with the case $r=3$. 

\begin{claim}\label{cl:sup-base}
For any $\epsilon<1$ we have $R^{3}_3(BK_t,BK_t,BK_t)\ge (t-1)t^{\epsilon}$ for $t$ sufficiently large.
\end{claim}
\begin{proof}
Let $\epsilon<1$. Take a vertex set consisting of $t-1$ disjoint sets of vertices $V_1, V_2, \ldots, V_{t-1}$, each of size $t^{\epsilon}$. If a hyperedge contains vertices from three different $V_i$, then color it green.  By the well-known lower bound on the diagonal Ramsey number $R(K_{t^{1-\epsilon}},K_{t^{1-\epsilon}}) =\Omega(2^{t^{1-\epsilon}/2})$, we can find a coloring of $K_{t-1}$ containing no clique of size $t^{1-\epsilon}$ when $t$ is sufficiently large.  Given such a red-blue coloring on the complete graph with vertex set $\{1,2,\dots,t-1\}$ we color the hyperedges consisting of two vertices from $V_i$ and one from $V_j$ by the color of $\{i,j\}$ in the graph.  We color every hyperedge completely contained in some $V_i$ red.  Observe that the core of any red or blue $BK_t$ may contain vertices in less than $t^{1-\epsilon}$ different classes and so has a total of less than $t$ vertices.
\end{proof}
\begin{remark}
This proof can give a slightly better bound on the order of $\frac{t^2}{\log(t)}$ but we write the bound in terms of $\epsilon$ for a simpler presentation.  
\end{remark}

\begin{theorem}
\label{colors}
For any uniformity $r \geq 4$, and sufficiently large $c$ and $t$, we have 
\begin{displaymath}
R^{r}_c(BK_t,BK_t,\dots,BK_t) > t^{1+ \left(\frac{r-3}{r-2}\right)^{r-3}-\left(\frac{r-3}{r-2}\right)^{r-2}}.
\end{displaymath}
\end{theorem}

Theorem \ref{colors} will follow from the following claim which we will prove by induction on $r$ by choosing the optimal $\epsilon$. 

\begin{claim}
\label{rec}
For any uniformity $r \geq 3$, and for any $\epsilon$ where $\epsilon<1$, for sufficiently large $c$ and $t$, we have
\begin{displaymath}
R^{r}_c(BK_t,BK_t,\dots,BK_t) >  t^{1+(1-\epsilon)^{r-3}-(1-\epsilon)^{r-2}}.
\end{displaymath}
\end{claim}

\begin{proof}
The base case follows from Claim~\ref{cl:sup-base}. Now assume that $r\geq 4$. Let $\epsilon< 1$. Let $c_s$ be the number of colors required for Claim~\ref{rec} to hold for an $s$-uniform hypergraph for $2 \le s \le r-1$. Let $M$ be the lower bound we obtain by induction for the function $R^{r-1}_{c_{r-1}}(BK_{t^{1-\epsilon}},BK_{t^{1-\epsilon}},\dots,BK_{t^{1-\epsilon}}).$  We will show
\begin{displaymath}
R^{r}_{c_{r}}(BK_t,BK_t,\dots,BK_t) > M \cdot t^\epsilon.
\end{displaymath}
for some constant $c_r$ depending on $r$.

Take the complete $r$-uniform hypergraph $\h$ on $N = M \cdot t^\epsilon$ vertices.  Partition the vertex set into sets $V_1,V_2,\dots,V_M$ each consisting of $t^\epsilon$ vertices. We consider $s$-uniform complete hypergraphs $\h_s$ defined on the vertex set $\{1,2,\dots,M\}$ for $2 \le s \le r-1$. Since the lower bounds in Claim \ref{rec} are decreasing (in $r$), we have for $c_s$ colors a coloring of $\h_s$ with no Berge clique of size $t^{1-\epsilon}$ provided $t$ is sufficiently large.  Assume, indeed, that $t$ is at least the maximum required for any $s$. 

Now, given the colorings of $\h_i$ with $c_i$ colors, we define a coloring on $\h$ with $c_r = \sum_{s=2}^{r-1} c_s+2$ colors and no monochromatic $BK_t$.  For $2 \le s \le r-1$ we color all hyperedges containing elements of the vertex sets $V_{i_1},V_{i_2},\dots,V_{i_s}$ with the same color as $\{i_1,i_2,\dots,i_s\}$ in the coloring of $\h_s$.  Observe that the core of a monochromatic $BK_t$ in $\h$ can contain vertices from fewer than $t^{1-\epsilon}$ classes. Since $\h_s$ has no monochromatic $BK_{t^{1-\epsilon}}$, and each class has $t^\epsilon$ vertices, it follows that $\h$ has no monochromatic $BK_t$ using hyperedges containing vertices from between 2 and $r-1$ classes. Finally, we may color the hyperedges contained in each $V_i$ with any color used so far and the hyperedges containing vertices from $r$ classes with a new color.  

It remains to verify that $M \cdot t^\epsilon$ yields the required bound.  Indeed,
\begin{displaymath}
M\cdot t^\epsilon =  t^{(1-\epsilon)\left(1+(1-\epsilon)^{r-4}-(1-\epsilon)^{r-3}\right)}\cdot t^\epsilon = t^{1+(1-\epsilon)^{r-3}-(1-\epsilon)^{r-2}}. \qedhere
\end{displaymath}

\end{proof}

We now discuss briefly the case of forbidding Berge-cliques of higher uniformity. First we collect some basic lemmas about the Ramsey number for Berge cliques in different uniformities.  

\begin{lemma}
\label{differentorders}
For any $r,c,a,b$, where $a<b$ and for $t$ sufficiently large, we have
\begin{displaymath}
R_c^r(BK_t^{(b)},BK_t^{(b)},\dots,BK_t^{(b)}) \ge R_c^r(BK_t^{(a)},BK_t^{(a)},\dots,BK_t^{(a)}). 
\end{displaymath}
\end{lemma}
\begin{proof}
It is sufficient to prove that for sufficiently large $t$, there is an injection from $\binom{[t]}{a}$ to $\binom{[t]}{b}$ mapping sets to one of their supersets.  Let $S \subset \binom{[t]}{a}$ and $\phi(S)$ be the elements of $\binom{[t]}{b}$ which contain some element from $S$.  We have $\abs{S}\binom{t-a}{b-a} \le \abs{\phi(S)} \binom{b}{a}$ by double-counting the relations between the two levels. Then $\abs{\phi(S)} \ge \abs{S}$ is obvious for sufficiently large $t$, and we have the desired injection by Hall's theorem.
\end{proof}


\begin{corollary}
\label{higherorder}
For any uniformity $r$, $a<r$, and sufficiently large $c$ and $t$, we have 
\begin{displaymath}
R^{r}_c(BK_t^{(a)},BK_t^{(a)},\dots,BK_t^{(a)}) \ge t^{1 + \left(\frac{r-3}{r-2}\right)^{r-3}-\left(\frac{r-3}{r-2}\right)^{r-2}}.
\end{displaymath}
\end{corollary}
\begin{proof}
The result is immediate from Lemma \ref{differentorders} and Theorem \ref{colors}.
\end{proof}

\section{Ramsey numbers of 2-shadow graphs and proof of Theorem \ref{thm:2-shadow}}
\label{2shadow}
In this short section, we discuss some results on the Ramsey number of $R^r(\partial K_t, \partial K_s)$. On the one hand, we have $R^r(\partial K_t, \partial K_s) \le R^r(BK_t, BK_s)$. Most of the constructions from Section \ref{sc:Berge} are also constructions for $R^r(\partial K_t, \partial K_s)$; however, there are some exceptions.

\begin{proposition}
For $s,t \ge 3$, we have  $R^3(\partial K_t, \partial K_s) = t+s-3$.  For $s \ge 3$, $R^3(\partial K_2, \partial K_s) = s$ and  $R^3(\partial K_2, \partial K_2) =3$.  
\end{proposition}
\begin{proof}
It is easy to see that $R^3(\partial K_2, \partial K_2) =3$ and $R^3(\partial K_2, \partial K_s) = s$ for $s\geq 3$. We will now show $R^3(\partial K_t, \partial K_s) \leq t+s-3$ for $s,t\ge 3$ by inducting on $s+t$.
The cases when $s$ or $t$ is 3 are trivial.  Assume the theorem holds for smaller $s+t$ and take a $2$-edge-colored complete $3$-uniform hypergraph $\cH$ on the vertex set $V$ of size $s+t-3$ where $s,t\geq 4$. If for all $x,y \in V$ we have that there exists $z$ such that $\{x,y,z\}$ is blue, then we have complete blue clique in the 2-shadow.  Otherwise suppose there is a pair of vertices $x,y$ such that for all $z \in V\setminus \{x,y\}$ we have $\{x,y,z\}$ is red, then consider the subhypergraph of $\cH$ induced by $V \setminus \{x\}$.  By induction, there exists either a blue $\partial K_t$, in which case we are done, or a red $\partial K_{s-1}$ with $Y$ as its core. Then we can extend it to a red $\partial K_s$ with $Y\cup \{x\}$ as its core by adding the red hyperedges $\{x,y,z\}$ where $z \in Y$.

The lower bound construction is to take a set of $t-2$ vertices $A$ and a set of $s-2$ vertices $B$ and color a hyperedge red if and only if it intersects $A$ in at most 1 vertex.
\end{proof}

\begin{proposition}
For $r \ge 4$ and $s,t \ge 2$, we have $R^r(\partial K_t, \partial K_s) = \max\{s,t,r\}$.  
\end{proposition}
\begin{proof}
Consider a $2$-edge-colored complete $r$-uniform hypergraph on $N= \max\{s,t,r\}$ vertices. 
Suppose first, that for all $x,y \in V$ there exists $z_1,z_2,\dots,z_{r-2}$ such that $\{x,y,z_1,z_2,\dots,z_{r-2}\}$ is blue, then there is a blue $K_N$ in the shadow.  
On the other hand, if there are $x,y \in V$, such that for all $z_1,z_2,\dots,z_{r-2}$, $\{x,y,z_1,z_2,\dots,z_{r-2}\}$ is red, then it is easy to see that there is a red $K_N$ in the 2-shadow. 
Thus, $R^r(\partial K_t, \partial K_s) \le \max\{s,t,r\}$. 
On the other hand taking a clique of the appropriate color on $\max\{s,t,r\}-1$ vertices yields a construction for the lower bound.
\end{proof}

\begin{remark}
The superlinear lower bounds constructed in Subsection \ref{superlinear} are in fact constructions for hypergraphs without monochromatic cliques in the 2-shadow.  Thus, the same lower bounds hold.
\end{remark}

\section{Ramsey numbers of trace-cliques}\label{sc:trace}

Throughout this section, we assume that $a,b$ are positive integers.

\begin{lemma}
\label{weak}
$R^{a+b+1}(T K^{(a+1)}_{t},T K^{(b+1)}_{s}) \leq R^{a+b+1}(T K^{(a+1)}_{t-1},T K^{(b+1)}_{s})+s-b$, for $t\geq a+1, s \geq b+1$. 

\end{lemma}

\begin{proof}
Let $N = R^{a+b+1}(T K^{(a+1)}_{t-1},T K^{(b+1)}_{s})+s-b$, and $\h$ be a $2$-edge-colored (blue and red) complete $(a+b+1)$-uniform hypergraph on $N$ vertices. Let $\h'$ be an induced subhypergraph of $\h$ on $R^{a+b+1}(T K^{(a+1)}_{t-1},T K^{(b+1)}_{s})= N-(s-b)$ vertices, obtained by removing a set $Y$ of $s-b$ vertices. Then $\h'$ contains either a blue  $T K^{(a+1)}_{t-1}$ or a red $T K^{(b+1)}_{s}$. 
In the second case we are done, so let us assume that $\h'$ contains a blue  $T K^{(a+1)}_{t-1}$ with core $X$. 
Let $Z$ be a set of $b$ vertices of $\h'$ which does not intersect $X$ (there is such a set since $v(\h') \geq v(T K^{(a+1)}_{t-1}) \geq t-1+b$) and put $S = Y \cup Z.$ 
Consider the edges of the form $A\cup B$ where $A\subseteq X, |X| = a$ and $B \subseteq S, |B| = b+1$. 
If for some fixed $B$, $A \cup B$ is blue for every subset $A$ of $X$ of size $a$, then pick $v \in B \cap Y$, and together with these edges and the edges defining the blue $T K^{(a+1)}_{t-1}$, $X \cup \{v\}$ is the core of a blue  $T K^{(a+1)}_{t-1}$. 
If this is not the case, then for any $B \subset S$ of size $b+1$, there exists $A_B \subseteq S$ such that $A_B \cup B$ is red, and therefore, $S$ together with these edges is the core of a red $T K^{(b+1)}_{s}$.
\end{proof}

\begin{theorem}
Let $t \geq a+1, s\geq b+1$. Then  $R^{a+b+1}(T K^{(a+1)}_{t},T K^{(b+1)}_{s}) \leq (t-a)(s-b) + a + b.$
\end{theorem}

\begin{proof}
We are going to prove this result by induction on $t$, the base case is where $t = a+1$, we have that $R^{a+b+1}(T K^{(a+1)}_{a+1},T K^{(b+1)}_{s}) = s+a = (s-b) + b + a$, so the result follows.
Now assume that for  some $t-1 \geq a+1 $ the result is true, then by Lemma~\ref{weak} we have that 

\begin{align*}
R^{a+b+1}(T K^{(a+1)}_{t},T K^{(b+1)}_{s}) &\leq R^{a+b+1}(T K^{(a+1)}_{t-1},T K^{(b+1)}_{s})+s-b\\  &\leq (t-1-a)(s-b) + a + b + (s-b) = (t-a)(s-b) + a + b. \qedhere \end{align*}
\end{proof}

\begin{proposition}
\label{tracesus}
Suppose that $t\geq a+1 \geq 3$ and $s\geq 2.$ Then
\begin{displaymath}R^{a+1}(S K^{(a)}_t,T K_s) \leq t + \max\{R^{a+1}(S K^{(a)}_{t-1},T K_s),R^{a+1}(S K^{(a)}_t,T K_{s-1})\}.
\end{displaymath}
\end{proposition}

\begin{proof}
Let $\h$ be an $(a+1)$-uniform hypergraph with vertex set $V$ of size 
\begin{displaymath}
N =  t + \max\{R^{a+1}(S K^{(a)}_{t-1},T K_s),R^{a+1}(S K^{(a)}_t,T K_{s-1})\}. 
\end{displaymath}
Since $N > R^{a+1}(S K^{(a)}_{t-1},T K_s)$, it follows that we can find either a blue $S K^{(a)}_{t-1}$ or a red $T K_s$. In the latter case, we are done, so assume there is a blue $S K^{(a)}_{t-1}$ with defining vertices $X$ and suspension vertex $u$.  Now, if for some $v \in V\backslash (X\cup\{u\})$ it holds that for every set $A \subseteq X$ of size $a-1$ we have that $A\cup\{v,u\}$ is blue, then we can add $v$ to $X$ and obtain a blue $S K^{(a)}_t$. Otherwise suppose that for every $v$ we can find a set $A_v$ such that $A_v \cup \{v,u\}$ is red. Let $V' =  V\setminus(X \cup \{u\})$. Note that $\abs{V'} \geq R^{a+1}(S K^{(a)}_t,T K_{s-1})\}$. It follows that we can find either a blue $S K^{(a)}_t$ or a red $T K_{s-1}$ in $\h[V']$. If we find a blue $S K^{(a)}_t$, we are done. Otherwise suppose we can find a red  $T K_{s-1}$ defined on the set $Y$. Then we can extend $Y$ to a red $T K_s$ by adding to $Y$ the vertex $u$ together with the edges $A_v\cup\{v,u\}$ for every $v \in Y$ since $A_v$ does not intersect $V'$. 
\end{proof}

\begin{corollary} Suppose that $t\geq a \geq 2$ and $s \geq 2$. Then
\begin{displaymath} R^{a+1}(S K^{(a)}_t,T K_s) \leq  \binom{t}{2}+(s-1)t.
\end{displaymath}
\end{corollary}
\begin{proof}
This bound follows by induction on $s+t$ from Proposition \ref{tracesus}. The case when $s=2$ or $t=a$ are trivial.  Assume we had the bound for smaller values of $s+t$ and observe that Proposition \ref{tracesus} and induction imply that
\begin{displaymath}
R^{a+1}(S K^{(a)}_t,T K_s) \le t + \max \left(\binom{t-1}{2} + (s-1)(t-1), \binom{t}{2}+(s-2)t\right) \le \binom{t}{2}+(s-1)t,
\end{displaymath}
as required.
\end{proof}

\begin{proposition}
Suppose that $t\geq a+1$ and $s\geq 2$. Then
\begin{displaymath} R^{a+1}(S K^{(a)}_t,\partial K_s) \ge (s-1)\floor{\frac{t}{a}}+1.
\end{displaymath}
\end{proposition}
\begin{proof}
Take a vertex set of size $(s-1)\floor{\frac{t}{a}}$ and divide it into $s-1$ classes $V_1,V_2,\dots,V_{s-1}$ of size at most $\floor{\frac{t}{a}}$.  Color every hyperedge which intersects each $V_i$ in at most 1 with red, and color every other hyperedge blue.  Clearly this construction has no red $\partial K_s$.  We will now show it has no blue $S K^{(a)}_t$. Indeed, suppose that $X$ is the core of the blue suspension and $v$ is the suspension vertex.  

Let $V_{i_1},\dots,V_{i_k}$ denote the classes which have nonempty intersection with $X \cup \{v\}$, then $t+1 = \abs{X\cup \{v\}} = \sum_{j = 1}^k \abs{(X \cup \{v\})\cap V_{i_j}} \leq \frac{kt}{a}$. It follows that $k > a$.  Suppose, without loss of generality, that $v \in V_{i_{a+1}}$.  Then we may take $x_j  \in X\cap V_{i_j}$ for $j = 1,\dots,a$ so that the edge $\{x_1,\dots,x_a,v\}$ is red, and thus not a member of a blue suspension, contradiction.
\end{proof}

Thus, we have the following corollaries.

\begin{corollary}
Suppose that $t\geq a+1$ and $s\geq 2$. Then
\begin{displaymath} R^{a+1}(S K^{(a)}_t,T K_s) \ge (s-1)\floor{\frac{t}{a}}+1.
\end{displaymath}
\end{corollary}

\begin{corollary}
$R^{a+1}(S K^{(a)}_t,T K_t)  = \Theta_{a}(t^2).$
\end{corollary}

\begin{proposition} Suppose that $t\geq a+2$ and $s\geq b+2$. Then
\begin{displaymath} R^{a+b+1}(HK^{(a+1)}_t,TK^{(b+1)}_s) \leq M + t+b\binom{t}{a+1} - b,\end{displaymath} where $M = \max \left(R^{a+b+1}(HK^{(a+1)}_{t-1},TK^{(b+1)}_s),R^{a+b+1}(HK^{(a+1)}_t,TK^{(b+1)}_{s-1})\right).$

\end{proposition}

\begin{proof}
Let $\h$ be an $(a+b+1)$-uniform hypergraph with vertex set $V$ of size 
\begin{displaymath}
N =  M + t+b\binom{t}{a+1} - b. 
\end{displaymath}
Since $N > M$, we can find either a blue $HK^{(a+1)}_{t-1}$ or a red $TK^{(b+1)}_s$. If the latter case occurs we are done, so assume there is a blue $H K^{(a+1)}_{t-1}$ with core $X$ of size $t-1$ and set of expansion vertices $X'$ of size $\binom{t-1}{a+1}b$. Now let $v$ be a vertex not in $X \cup X'$. We will try to extend $X$ together with $v$. Let $A_1,A_2,\dots,A_{\binom{t-1}{a}}$ be an ordering of the subsets of $X$ of size $a$. Let $V_1 = V\backslash(X \cup X' \cup \{v\})$ and set $X_1 = X'$. For each $i = 1, 2\dots,\binom{t-1}{a}$, if there is a set $B_i$ of size $b$ in $V_i$ such that that $B_i \cup A_i \cup \{v\}$ is blue, then set $V_{i+1} = V_i \backslash B_i$ and $X_{i+1} = X_i \cup B_i$, otherwise we stop. If we can do this for every $i$ then the set $X\cup\{v\}$ defines a blue $HK^{(a+1)}_t$ with expansion set $X_{\binom{t-1}{a}}$. If not, then there is an index $i$ such that we have to stop. This means that for every set $B$ of size $b$ in $V_i$ we have that $A_i \cup B \cup \{v\}$ is red. Now the size of $V_i$ is $N-(t-1) - \binom{t-1}{a+1}b - (i-1)b -1 \geq  N-t - \binom{t-1}{a+1}b - (\binom{t-1}{a}-1)b = M.$ So by the definition of $M$, we can find either a blue $HK^{(a+1)}_t$ using $V_i$ or a red $TK^{(b+1)}_{s-1}$. In the first case we are done, so suppose we have a red $TK^{(b+1)}_{s-1}$ with defining vertices $Y$. Now we can extend $Y$ together with $v$ to a red $TK^{(b+1)}_s$, since for every $B \subseteq Y$ of size $b$ we have that the edge $B\cup A_i \cup \{v\}$ is red.
\end{proof}

\begin{corollary}
Suppose that $t\geq a+1$ and $s\geq b+1$. Then
\begin{displaymath}
R^{a+b+1}(HK^{(a+1)}_t,TK^{(b+1)}_s) \le b \binom{t+1}{a+2}+\binom{t+1}{2} -tb+s\left(b\binom{t}{a+1}+t-b\right).
\end{displaymath}
\end{corollary}

\section{Ramsey number of expansion and suspension hypergraphs} \label{sc:exp-susp}

\subsection{Expansion hypergraphs and Proof of Theorem \ref{th:expansion-upper}}

In this section, we give an upper bound on $R^3(HK_{t}, HK_s)$. Recall that $H^r(K_t)$ is the the $r$-graph obtained from the complete graph $K_t$ by enlarging each edge by a set of $(r-2)$ distict new vertices.
Moreover, $R^r(H^r(K_t), H^r(K_t))$ is the smallest integer $n$ such that every $2$-edge-coloring of the complete $r$-uniform hypergraph $\cH$ on $n$ vertices contains a monochromatic $H^r(K_t)$. For ease of reference, we will use $R^r(HK_t, HK_t)$ to denote $R^r(H^r(K_t), H^r(K_t))$.
We first prove the following lemma.

\begin{lemma}\label{th:ex:induct}
For $s,t \geq 2$, we have that
\begin{displaymath}
R^3(HK_{t+1}, HK_{s+1}) \leq \max\{R^3(HK_{t+1},HK_s), R^3(HK_{t},HK_{s+1})\} + 2st.
\end{displaymath}
\end{lemma}
\begin{proof}
Without loss of generality, we assume that $t\leq s$.
Let
\begin{displaymath}
N = \max\{R^3(HK_{t+1},HK_s), R^3(HK_{t},HK_{s+1})\} + 2st 
\end{displaymath}
and  $\cH_N$ be a $2$-edge-colored compete $3$-uniform hypergraph on $N$ vertices. 
Let 
\begin{displaymath}
W = \{v_1, v_2, \ldots, v_{2st}\} \subset V(\cH_N)
\end{displaymath}
and $\cH' = \cH[V(\cH_N) \backslash W]$.

Note that $|\cH'| \geq R^3(HK_{t}, HK_{s+1})$. Thus by definition of Ramsey number, there exists either a blue expansion of $K_{t}$ or a red expansion of $K_{s+1}$. If the latter happens, we are done. Thus assume that we have a blue expansion $\cH_b$ of $K_t$. Note that $\cH_b$ has $\binom{t}{2} + t$ vertices. Let $\{u_1,\ldots u_t\}$ be the core of $\cH_b$. Let $F= V(\cH)\backslash V(\cH_b)$.

\begin{claim}\label{cl:ex:induct}
Suppose that $\cH_N$ does not have a blue expansion of $K_{t+1}$. Then for every $v \in W$, there exists some $u$ in the core of $\cH_b$ such that $\{v,u,w\}$ is colored red for all $w$ except at most $(t-1)$ elements from $F\backslash\{v\}$.
\end{claim}

\begin{proof}	
Fix a vertex $v \in W$. Construct a bipartite graph $G= A \cup B$ where $A = \{u_1, \ldots, u_t\}$ and $B =F\backslash\{v\}$. For $u_i\in A$, $w \in B$, $u_i$ is adjacent to $w$ in $G$ if and only if $\{v,u_i,w\}$ is a blue edge in $\cH_N$.  Note that for every $w \in B$, $d_G(w) \leq t$. Therefore, if $d_G(u_i) \geq t$ for every $u_i \in A$, then there exists a matching of $A$ in $G$ by Hall's theorem, which implies that we can extend $\cH_b$ to a blue expansion of $K_{t+1}$ by adding $v$ into the core of $\cH_b$. This contradicts our assumption that $\cH_N$ does not have a blue expansion of $K_{t+1}$. Hence it follows that there exists a vertex $v' \in A$ such that $\{v,v',w\}$ is colored red for all except $t-1$ elements of $F\backslash\{v\}$. This finishes the proof of Claim~\ref{cl:ex:induct}
\end{proof}

Now since $\abs{W}  = 2st$, by pigeonhole principle, there exists some $u$ in the core of $\cH_b$ so that there exists $W_u = \{w_1, w_2, \ldots, w_s\}$ such that for any $w \in W_u$, the hyperedge $\{w,u, w'\}$ is red for all $w'$ except at most $(t-1)$ elements of $F\backslash\{w\}$. Let $M(w_i)$ be the elements $w'$ in $W$ such that $\{u,w_i, w'\}$ is blue. 

Now let $W' = W_u \cup V(\cH_b) \cup \bigcup_{i=1}^s M(w_i)$ and $\cH'' = \cH_N[V(\cH_N) \backslash W']$. Note that\\ $|\cH''| \geq R^3(HK_{t+1}, HK_s)$ since $2st \geq st + \binom{t}{2}+t$. Hence there either exists a blue expansion of $K_{t+1}$ or there exists a red expansion of $K_s$. If the former happens, we are done. Hence assume we have a red expansion $\cH_r$ of $K_s$. 
Suppose $\{v_1, v_2, \ldots, v_s\}$ is the core of $\cH_r$.
Now we can extend $\cH_r$ to a red expansion of $K_{s+1}$ by adding $u$ into the core of $\cH_r$ together with the red edges in $\{\{u,w_i, v_i\}: i\in [s]\}$. This completes the proof of the lemma.
\end{proof}

Now we are ready to show that $R^3(HK_t, HK_s) \leq 2(s+t) st$. The proof is by induction on $s + t$. We first show that $R^3(HK_2, HK_s) \leq 4s^2 + 8s$. This is clearly true since any blue edge in a 3-uniform hypergraph is a blue expansion of $K_2$. Hence given any $2$-edge-colored complete 3-uniform hypergraph $\cH$ with $4s^2 + 8s$ vertices, if there is no blue edge, then all edges are red, which implies that we have a red expansion of $K_s$, since $4s^2 + 8s \geq \binom{s}{2}+s$. Similarly, $R^3(HK_t, HK_2) \leq 4t^2 + 8t$.

Now assume the theorem holds for $HK_{t'}, HK_{s'}$ such that $t'+s' < t+s$. Without loss of generality, assume that $t\leq s$. Then by the Lemma~\ref{th:ex:induct}, 
\begin{align*}
R^3(HK_{t}, HK_{s}) &\leq \max\{R^3(HK_{t},HK_{s-1}), R^3(HK_{t-1},HK_{s})\} + 2(s-1)(t-1)\\
						    &\leq 2(s+t-1)t(s-1) + 2(s-1)(t-1)\\
						    &\leq 2st(s+t).
\end{align*}
Hence we are done by induction.

\subsection{Ramsey number of suspension hypergraphs}

Recall that the $r$-suspension $SK_t$ of the complete graph $K_t$, is the $r$-uniform hypergraph formed by adding a single fixed set of $r-2$ distinct new vertices to every edge in $K_t$.
Clearly, $R^{r}(SK_t, SK_t) \leq R^2(K_t,K_t)+(r-2)$. The proof is simple: let $\cH$ be a $2$-edge-colored $K^{(r)}_{R^2(K_t, K_t)+(r-2)}$. Fix a set of $(r-2)$ vertices $S$ and consider the complete graph $G$ on the remaining $R^2(K_t, K_t)$ vertices, where the color of an edge $e$ in $G$ is the same color as the hyperedge $e\cup S$ in $\cH$. By the definition of the Ramsey number, there exists a monochromatic clique in $G$, which gives us the core of the monochromatic $SK_t$ in $\cH$. 

Before we prove the lower bound, let us recall the symmetric version of the Lov\'{a}sz local lemma~\cite{Alon-Spencer}:

\begin{quotation}\it
  Let ${\mathcal  {A}}=\{A_{1},\ldots ,A_{q}\}$ be a finite set of events in the probability space $\Omega$. Suppose that each event $A_i$ is mutually independent of a set of all but at most $d$ of the other events
$A_j$, and that $\Pr(A_i) \leq p$ for all $1\leq i\leq q$. If 
$$ep(d+1)<1,$$
then $$\Pr \lp \displaystyle\bigwedge_{i=1}^q \overline{A_i}\rp > 0.$$
\end{quotation}

Now we can show a lower bound of $R^{r}(SK_t, SK_t)$ with the local lemma.

\begin{proposition}
Fix $t\geq r \geq 3$. If $$e \lp 1 + \binom{t}{2} \binom{r}{2} \binom{n}{t-2}\rp 2^{1-\binom{t}{2}} <1,$$
then $R^r(SK_t, SK_t) > n$.
\end{proposition}
\begin{proof}
Let $\cH$ be a complete $r$-uniform hypergraph on $n$ vertices. Color each hyperedge blue or red randomly and independently with probability $\frac{1}{2}$. For a set of $r-2$ vertices $S$ and another set of $t$ vertices $T$ disjoint from $S$, let $A_{S,T}$ be the event that the suspension hypergraph $\cH_{S,T}$ with $T$ as core and $S$ as the suspending vertex set is monochromatic. Note that for each fixed $S,T$, $$\Pr(A_{S,T}) = 2^{1-\binom{t}{2}} = p.$$

Note that $A_{S,T}$ is mutually independent of all other events $A_{S',T'}$ satisfying \\ $E(\cH_{S,T}) \cap E(\cH_{S',T'}) = \emptyset$. Let us give an upper bound on the number of events $A_{S',T'}$ that $A_{S,T}$ is mutually dependent of. There are $\binom{t}{2}$ choices to pick an edge they share, which contains $r$ vertices.
Among the $r$ vertices, $r-2$ of them must be the suspension vertices. There are $\binom{r}{r-2}$ ways to choose the suspension vertices $S'$. There are then at most $\binom{n}{t-2}$ ways to choose the remaining $t-2$ vertices of $T$. Hence it follows that 
$$d \leq \binom{t}{2} \binom{r}{2} \binom{n}{t-2}.$$

By the Lov\'{a}sz local lemma, it follows then that if $ep(d+1) <1$, we have that 
$$\Pr\lp \displaystyle\bigwedge_{S,T} \overline{A_{S,T}} \rp > 0.$$
Hence there exists a coloring of $\cH$ without any monochromatic $SK_t$.
\end{proof}
\begin{remark}
For any fixed $r$, this gives asymptotically the same lower bound as Ramsey number $R^2(K_t,K_t)$, i.e. $R^r(SK_t, SK_t) > (1+o(1))\frac{\sqrt{2}}{e} t \sqrt{2}^t$.
\end{remark}

\section{Acknowledgements}

We would like to thank J\'er\^ome  Argot and Xueping Zhao for their generous help providing the computing resources. The third author would like to thank Gyula O.H.~Katona for his hospitality and guidance during the third author's visit in Budapest. We also like to thank the anonymous referee for their useful suggestions.  The research of the first and second authors was partially supported by the National Research, Development and Innovation Office, NKFIH, grant K116769. The research of the first author was supported in part by NSF grant DMS-1600811. The research of the second author was supported in part by IBS-R029-C1.

\end{document}